\documentclass{amsart}
\usepackage{graphicx} 

  \usepackage[utf8]{inputenc}
  \usepackage{amssymb,amsthm,amsmath}
  \usepackage[hidelinks]{hyperref}
  
  \usepackage{graphicx}
  \usepackage{algorithm, tabularx}
  \usepackage[noend]{algorithmic}
  \usepackage{xcolor}
  \usepackage{csvsimple}
  \usepackage{filecontents}
  \usepackage{pgfplots}
  \usepackage{caption}
  \usepackage{rotating}
  \usepackage{verbatim}

\usetikzlibrary{pgfplots.groupplots}
\pgfplotsset{compat=1.18,
    every non boxed x axis/.style={} 
}
  \theoremstyle{plain}
  \newtheorem{thm}{Theorem}[section]
  \newtheorem{lem}[thm]{Lemma}
  \newtheorem{cor}[thm]{Corollary}
    \newtheorem{prop}[thm]{Proposition}

  \theoremstyle{definition}
  \newtheorem{defn}[thm]{Definition}
  
\newtheorem{exmp}{Example}[section]

\theoremstyle{remark}

\def\hh{{\mathbf h}}

\newcommand{\bigO}{{O}}

\newcommand{\dvs}{{\rm div}}
\newcommand{\algorithmicbreak}{\textbf{break}}

\newcommand\norm[1]{\left\lVert#1\right\rVert}
\newcommand\infnorm[1]{\left\lVert#1\right\rVert_\infty}
\newcommand\td{{\varepsilon}}

\newcommand\rr[1]{\mathsf{RR}\left(#1\right)}
\newcommand\aug{\fboxsep=-\fboxrule\!\!\!\fbox{\strut}\!\!\!}

\definecolor{blue2}{rgb}{0.36, 0.54, 0.66}
\definecolor{red2}{rgb}{1.0, 0.13, 0.32} 
\definecolor{violet2}{rgb}{0.53, 0.38, 0.56}

\title[Solving Norm Equations in Global Function Fields]{Solving Norm Equations in Global Function Fields}

\author{Sumin Leem}
\address{Department of Mathematics and Statistics\\
	University of Calgary, 2500 University Drive NW\\
	Calgary, AB T2N 1N4, Canada}
\email{sumin.leem0@gmail.com}

\author{Michael J. Jacobson, Jr.}
\address{Department of Computer Science\\
	University of Calgary, 2500 University Drive NW\\
	Calgary, AB T2N 1N4, Canada}
\email{jacobs@ucalgary.ca}

\author{Renate Scheidler}
\address{Department of Mathematics and Statistics\\
	University of Calgary, 2500 University Drive NW\\
	Calgary, AB T2N 1N4, Canada}
\email{rscheidl@ucalgary.ca}

\begin{document}

\begin{abstract}
We present two new algorithms for solving norm equations over global function fields with at least one infinite place of degree $1$ and no wild ramification. The first of these is a substantial improvement of a method due to Ga\'{a}l and Pohst, while the second approach uses index calculus techniques and is significantly faster asymptotically and in practice. Both algorithms incorporate compact representations of field elements which results in a significant gain in performance compared to the Ga\'{a}l-Pohst approach. We 
provide Magma implementations, analyze the complexity of all three algorithms under varying asymptotics on the field parameters, and provide empirical data on their performance.    
\end{abstract}

\keywords{Primary 11Y16; Secondary 11Y40, 11R58, 11D57, 11G20}

\maketitle

\section{Introduction}

Solving norm equations, which are a special instance of Diophantine equations over global fields, is a classical research area in number theory. In the setting of number fields, this problem has undergone extensive investigation. In 1973, Siegel proposed a method for solving norm equations in Galois fields by bounding the absolute values of the solutions  \cite{SiegelNENF}. In 1989, Pohst and Zassenhaus introduced a technique for solving norm equations over algebraic number fields by exhaustive search \cite{PohstZassenhaus}, using inequalities given in \cite{MahlerinequalitiesIB}. In \cite{fiekerjurkpohst}, Fieker, Jurk, and Pohst presented an exhaustive search algorithm for solving relative norm equations that uses a modified version of the Finke-Pohst enumeration algorithm of \cite{FinckePohst}. Finally, Simon developed a way to solve norm equations algebraically using $S$-units in 2002 \cite{SimonNENF}. In contrast,
solving norm equations in algebraic function fields has undergone far less exploration. The only algorithm in the literature to date, proposed in 2009, is due to Ga\'{a}l and Pohst \cite{GaalPohstnorm}, who adapted the exhaustive search method of \cite{PohstZassenhaus} to function fields. 

One of the difficulties arising in solving norm equations in any global field is the size of their solutions. An illustrative example of this behaviour is exhibited by algebraic units which can be extremely large, yet their norm is very small. An innovative way to address this problem is to represent the solutions, as well as other field elements arising in the computation of these solutions, in compact representation. 
In global function fields, 
the concept of compact representations was first introduced in 1996 by Scheidler \cite{ScheidlerCR} in the setting of quadratic extensions. In that source and a subsequent paper  \cite{ScheidlerCRApp}, the concept was used to prove membership in NP and other complexity results for several decision problems arising in computational number theory. 
%
In 2013, Eisentr{\"a}ger and Hallgren generalized the concept of compact representation to arbitrary global function fields \cite{EisentragerHallgren}. They provided a proof that the principal ideal problem in this setting belongs to NP and presented polynomial time quantum algorithms for computing a generator of a principal ideal, the unit group, and the class group of a global function field.


We present two novel algorithms for solving norm equations in global function fields using compact representations; one is an exhaustive search algorithm inspired by the existing method by Ga\'{a}l and Pohst in \cite{GaalPohstnorm}, and the other is based on principal ideal testing via index calculus.  Our algorithms require that the function field have at least one infinite place of degree 1 and no wild ramification.  We provide detailed complexity analyses of our new algorithms as well as the Ga\'{a}l-Pohst algorithm and the algorithm for computing compact representations of Eisentr\"ager and Hallgren \cite{EisentragerHallgren}.  All these algorithms were implemented in Magma \cite{magma}, and we performed extensive numerical experiments
to compare their performance in practice. Our Magma implementation, including the testing code, can be found in \cite{github_Sumin}. 

Incorporating compact representations into the algorithms not only allowed for smaller representations of solutions, but also led to a significant speed-up 
compared to the Ga\'{a}l-Pohst method that does not use compact representations.  Our new algorithms outperformed the Ga\'{a}l-Pohst algorithm enormously in terms of run time and have significantly better asymptotic complexities, exponentially better in terms of most of the main function field parameters. The algorithm using 
index calculus 
was the most efficient both in terms of asymptotic complexity and in practice, especially for large inputs.

\section{Notation and Preliminaries}

For background and an algebraic treatment of global function fields, the reader is referred to \cite{StichtenothBook, RosenNTFF}.

\subsection{Global function fields and norm equations}\label{sec:FFandNE}

Let $k=\mathbb{F}_q$ be a finite field of $q$ elements and $x$ be transcendental over $k$. Denote the polynomial ring and the rational function field over $k$ in $x$ by $k[x]$ and $k(x)$, respectively. A global function field $F$ is a finite algebraic extension of $k(x)$. Let $n$ denote its extension degree over $k(x)$ and $g$ its genus. We assume that $\gcd(n,q) = 1$ to avoid wild ramification. 

The places of $F$ are partitioned as $P(F) = P_\infty(F) \cup P_0(F)$ where $P_\infty(F)$ is the set of \emph{infinite places} of $F/K(x)$ (consisting of the poles of $x$), and $P_0(F)$ is the set of \emph{finite places} of $F/k(x)$ (corresponding to the non-zero prime ideals in $O_F$). 

The \emph{maximum norm} of of any element $\alpha \in F^\times$ is defined to be 
\begin{equation*}\label{def:maxnorm} 
    \infnorm{\alpha} = \max_{P\in {P}_\infty(F)} \{-v_P(\alpha)/e_P \}, 
\end{equation*}
where $v_P()$ denotes the discrete valuation corresponding to a place $P \in P(F)$ and~$e_P$ is the ramification index of $P$ in $F/k(x)$.

Henceforth, we fix a \emph{reduced} basis of $F/k(x)$, 
i.e.\ a basis $\mathcal{B} = \{ \omega_1, \ldots , \omega_n\}$ that is also a $k[x]$-basis of the maximal order $O_F$ such that $\infnorm{\omega_1} \le \cdots \le \infnorm{\omega_n}$ and 
\[ \infnorm{ \lambda_1 \omega_1 + \cdots + \lambda_n \omega_n} = \max_{1 \le i \le n} \infnorm{\lambda_i \omega_i} \]
for for all $\lambda_1, \ldots , \lambda_n \in k(x)$, not all zero.

A \emph{norm equation} in $F/k(x)$ is an identity of the form 
\begin{equation}\label{eq:ne_defn}
    \textrm{Norm}_{F/k(x)}(\alpha)=c,
\end{equation}
where $\alpha \in O_F$ and $c\in k(x)\setminus \{0\}$. Solving a norm equation \eqref{eq:ne_defn} refers to finding all $\alpha \in O_F$ up to associates, i.e.\ factors that are units in $O_F$, such that $\textrm{Norm}_{F/k(x)}(\alpha) \in ck^\times$. By slight abuse of terminology, a solution of \eqref{eq:ne_defn} will mean a solution of $\textrm{Norm}_{F/k(x)}(\alpha)=\zeta c$ for any $\zeta \in k^\times$. Testing whether two elements $\alpha, \beta \in O_F$ are associate is accomplished by identifying all the places $P \in P_0(F)$ for which $v_P(\alpha) \ne 0$ or $v_P(\beta) \ne 0$, and then checking that $v_P(\alpha) = v_P(\beta)$ for all these places.

\subsection{$S$-units and their lattices} \label{ss:unit-lattices}

For any set $S \subseteq P(F)$, let $O_S$ denote the ring of $S$-integers and $O_S^\times$ the group of $S$-units.
For $S = P_\infty(F)$, we have $O_S = O_F$. 
%
Assume now that $S$ is finite with $P_\infty(F) \subseteq S \subset P(F)$. Fixing an order on the places $P_1, \ldots , P_{|S|}$ of $S$,
consider the two group homomorphisms
\begin{align*}
    \phi_{S} :F^\times \longrightarrow \mathbb{Z}^{\vert S \vert }, \qquad & \alpha \mapsto \left(-v_{P_1}(\alpha), \dots , -v_{P_{\vert S\vert}}(\alpha)\right) , \\
     \Phi_{S} :F^\times \longrightarrow \mathbb{Z}^{\vert S \vert }, \qquad & \alpha \mapsto \left(-v_{P_1}(\alpha) \deg P_1, \dots , -v_{P_{\vert S\vert}}(\alpha)\deg P_{|S|} \right).\label{map:Sunitlattice}
\end{align*}
The 
images $\Lambda_S' = \phi(O_S^\times)$ and $\Lambda_S = \Phi_S(O_S^\times)$ are the
\emph{$S$-unit valuation lattice} and the \emph{$S$-unit lattice} of $F/k(x)$, respectively; they are lattices over $\mathbb{Z}$ of rank $|S|-1$. 
Put $R_S' = \det \Lambda_S'$ and $R_S = \det \Lambda_S$,
where we write $R_S = R_F$ for $S=P_\infty(F)$ 
and refer to $R_F$ as the \emph{regulator} of $F/k(x)$.

Let $Cl(F)$ denote the class group of $F$. For any divisor $D$, let $[D]$ denote its class in $Cl(F)$. Then the kernel of the map
\begin{equation} \label{map:Psi_S}
    \Psi_S: \mathbb{Z}^{|S|} \rightarrow Cl(F) \quad \mbox{via} \quad (v_1, \ldots , v_{|S|}) \mapsto \left [ \sum_{i=1}^{|S|} v_i P_i \right ] 
\end{equation}
is isomorphic to 
$\Lambda_S'$, so basis of $\Lambda_S'$ by computed via a basis of $\ker(\Psi_S)$, computed in Algorithm \ref{alg:Svalmat}. 

\begin{algorithm}
    \caption{SValMat}
    \label{alg:Sunitlat} \label{alg:Svalmat}
\begin{algorithmic}[1]
\REQUIRE a finite set $S=\{P_1, \dots, P_{\vert S\vert}\}$ of places including all infinite places of $F$
\ENSURE An $S$-unit value matrix matrix $M_S \in  \mathbb{Z}^{(\vert S\vert -1) \times \vert S \vert}$. 
    \STATE Compute the structure of $Cl(F)$ using \cite[Algorithm 5.5]{hessphd} or the technique of \cite{CDiem_classgp_q}.
    \STATE $\Psi_S\leftarrow$ the map defined in (\ref{map:Psi_S}) 
    \STATE $\{v_1, \dots, v_{\vert S\vert}\}  \leftarrow$ Generators of $\ker(\Psi_S)$
    \STATE $M_0\leftarrow $ $(\vert S\vert -1) \times \vert S \vert$ matrix whose $i$-th row is $v_i$ for $1\le i \le \vert S\vert-1$
    \STATE $M_S\leftarrow$ LLL-reduction on $M_0$
    \RETURN $M_S$
    \end{algorithmic}
\end{algorithm}


\begin{lem}\label{lem:comp_SValMat}
The cost of Algorithm \ref{alg:Svalmat} is subexponential in $g$, and polynomial in $q$ and $n$.
\end{lem}
\begin{proof}
The cost of Algorithm \ref{alg:Svalmat} is dominated by step 1. When $g\rightarrow \infty$, we use Hess's relation search algorithm \cite[Algorithmus 5.5]{hessphd} whose expected run time is subexponential in $g$ under a reasonable smoothness assumption  \cite[Glattheitsannahme 4.19 and Satz 5.23]{hessphd}. 
%
%
When $q \rightarrow \infty$, Diem's method \cite{CDiem_classgp_q} computes $Cl^0(F)$ in expected time 
that is polynomial in $q$. Finally, class group computation is polynomial in $n$ when $n \rightarrow \infty$ and $g$, $q$ are considered fixed.
\end{proof}

For analyzing the size of the output of Algorithm \ref{alg:Svalmat}, we consider the standard maximum norm of any matrix $M = (m_{ij})$ over $\mathbb{Z}$, given by
\[ 
\infnorm{M} = \max_{i,j} |m_{ij}|.\]
(This should not be confused with the maximum norm on $F$; the context makes it clear which norm is under consideration.)

%

\begin{prop}\label{prop:MatrixMS_entry_bound}
The output $M_S$ of Algorithm \ref{alg:Svalmat} satisfies
\[ \infnorm{M_S} \le 2^{(\vert S \vert -1)(\vert S \vert -2)/4}R_{S}' .\]
\end{prop}
\begin{proof}
The rows 
of $M_S$ form an LLL-reduced basis of $\Lambda_S'$. Denoting by $(M_S)_j$ the $j$-th row of $M_S$, 
we have
\begin{equation*}
    \prod_{j=1}^{\vert S \vert -1} \norm{(M_S)_j} \le 2^{(\vert S \vert -1)(\vert S \vert -2)/4}\det\Lambda_S' = 2^{(\vert S \vert -1)(\vert S \vert -2)/4} R_S'
\end{equation*}
by \cite[Proposition 1.6]{LLL_original}. Chose $j$ such that $\infnorm{M_S}$ is taken on by an element in the $j$-th row. Then 
\[ \infnorm{M_S} = \infnorm{(M_S)_j} \le  \prod_{j=1}^{\vert S \vert -1} \infnorm{(M_S)_j}.\qedhere \]
%
\end{proof}

\begin{cor}\label{cor:bound_LLLbasis_of_Unit_lattice}
Let 
$\mathcal{B}_S=\{b_1, \dots, b_{\vert S\vert-1} \}$ be a LLL-reduced basis of 
$\Lambda_S$. Then
\[
\max_{1\le i \le \vert S\vert-1} \norm{b_i}_\infty \le 2^{(\vert S \vert -1)(\vert S \vert -2)/4}R_{S}.
\]
\end{cor}

By \cite[Prop.\ 14.1 (a)]{StichtenothBook}, we have $R_S' = |Cl^0(F)/Cl_S(F)| \le Cl^0(F)$, where $Cl_S(F)$ is the $S$-class group of $F$, i.e.\ group of divisors supported only outside $S$ modulo principal divisors. 
Using the Hasse-Weil bound $|Cl^0(F)| \le (\sqrt{q}+1)^{2g}$, we obtain 
\begin{equation}\label{ineq:RS'boundq^2g}
    R_S' \le R_F\le (\sqrt{q}+1)^{2g}.
\end{equation}

\section{Compact representation} 

An alternative to a standard representation of $\alpha \in F$, i.e.\ given in terms of a $k(x)$-basis of $F$, is to write $\alpha$ as a tuple of small elements in $F$ such that a suitable power product of these elements evaluates to $\alpha$. Such a compact representation is particularly well suited for elements $\alpha$ of small norm and a highly useful tool for solving norm equations. We follow the treatment of this subject in \cite{EisentragerHallgren} 
and assume throughout this section that $F/k(x)$ has an infinite place of degree 1. 

Write $P_\infty(F) = \{ P_{\infty_1}, \ldots, P_{\infty,r+1} \}$ with $\deg P_{\infty,r+1} = 1$. It will be helpful to define for any $\alpha \in F$ the $r$-tuple 
\[ \text{val}_\infty(\alpha) = (v_{P_{\infty,1}}(\alpha) , \ldots , v_{P_{\infty,r}}(\alpha)) \in \mathbb{Z}^r . \]
Let $v  = (v_1, \ldots , v_r) \in \mathbb{Q}^r$. Then $\alpha$ is said to be \emph{close} to $v$ if
%
%
\[ \sum_{i=1}^r \lvert v_{P_{\infty,i}}(\alpha) - v_i \rvert \le r+g. \]

Riemann-Roch spaces and minima play a key role in computing compact representations. The \emph{Riemann-Roch space} of a divisor $D$ of $F$ is the finite dimensional $k$-vector space
\[ L(D) := \{\alpha \in F^\times \vert \dvs(\alpha) \ge -D\} \cup \{0\} . \]
For a fractional $O_F$-ideal $I$, the \emph{divisor of $I$} is $\text{div}(I) = \sum_{P \in P_0(F)} n_P P$, where $n_P = v_{\mathfrak{p}}(I)$ and $\mathfrak{p}$ is the $O_F$-prime ideal corresponding to the place $P \in P_0(F)$. A non-zero element $\mu \in I$ is a \emph{minimum of $I$} if the following hold. If $\alpha \in I$ is non-zero such that $v_P(\alpha) \ge -v_P(\mu)$ for all $P \in P_\infty(F)$, then either $\alpha = 0$ or $v_P(\alpha) = v_P(\mu)$ for all $P \in P_\infty(F)$. In other words, if $D = \dvs(I)-\sum_{P \in P_\infty(F)} v_P(\mu)P$, then every $\alpha \in L(D)$ is either 0 or $v_P(\alpha) = v_P(\mu)$ for all $P \in P_\infty(F)$. A fractional $O_F$-ideal $I$ is \emph{reduced} if 1 is a minimum of $I$. For any fractional $O_F$-ideal $I$ and any minimum $\mu$ of $I$, the ideal $(\mu^{-1})I$ is reduced. For a reduced ideal $I$, we have $0 \le \deg \dvs(I) \le g$.


We now have all the ingredients to introduce compact representations. 

\begin{defn}{\cite[Definition 4.3]{EisentragerHallgren}} \label{def:CR}
A \emph{compact representation} of $\alpha \in F$ is a pair $\mathbf{t}_\alpha = \big (\mu, (\beta_1, \beta_2, \cdots, \beta_l) \big )$ where 
\begin{itemize}
    \item $l \le \log(\infnorm{\text{val}_\infty (\alpha)}+g)$,
    \item $\beta_1, \ldots \beta_l \in F$ such that $\displaystyle \beta= \prod_{i=1}^l \beta_i ^{2^{l-i}}$ is a minimum of $O_F$, 
    \item $\mu \in F$ satisfies $\alpha = \mu/\beta$,
    \item The number of bits required to represent $\mu$ is polynomial in $\log q$, $n$ and $\deg(\textrm{Norm}_{F/k(x)}(\alpha))$;
    \item The number of bits required to represent each $\beta_i$ is polynomial in $\log q$ and~$n$.
\end{itemize}
\end{defn}
This implies in particular that given a compact representation $\mathbf{t}_\alpha = (\mu,  (\beta_1, \dots, \beta_l))$ of $\alpha \in F$, we have
\begin{equation*}\label{eq:defn_comprep}
\alpha = \mu \prod_{i=1}^{l} \left(\frac{1}{\beta_i}\right)^{2^{l-i}}, 
\end{equation*}
where $l, \mu, \beta_1, \ldots , \beta_l$ are all small.

Eisentrager and Hallgen provided an algorithm for computing a compact representation of $\alpha \in F$ in \cite{EisentragerHallgren} which we reproduce here in more streamlined form. The algorithm first finds $\mu$ by computing a basis of a suitable Riemann-Roch space (Algorithm \ref{alg:reduce}) and then computes $\beta_1, \ldots , \beta_l$ via a square-and-multiply approach similar to binary exponentiation (Algorithm \ref{alg:comprep}). 

\begin{algorithm}[H]
    \caption{Reduce  \cite[Algorithm 3.4]{EisentragerHallgren}}\label{alg:reduce}
\begin{algorithmic}[1] 
\REQUIRE A fractional $O_F$-ideal $I$ and a vector $v=(v_1, v_2, \cdots, v_{r})\in \mathbb{Z}^r$
\ENSURE a minimum $\gamma$ of $I$ that is close to $v$ 
    \STATE $D\leftarrow \dvs(I)+\sum_{i=1}^{r} v_i P_{\infty,i}$
    \FOR {$\ell \in [-\deg D , -\deg D + g]$}
    \IF {$L(D+\ell P_{\infty,r+1})$ is not trivial}
    \STATE Pick $\gamma$ from a basis of $L(D+\ell P_{\infty,r+1})$.
    \RETURN $\gamma$
    \ENDIF
    \ENDFOR 
    \end{algorithmic}
\end{algorithm}

To find $\mathbf{t}_\alpha$, we first use Algorithm \ref{alg:reduce} to find a minimum $\mu$ close to 0 in $A = \alpha O_F$. Let $\mathbf{v} = \log_2\infnorm{\text{val}_\infty(\mu)-\text{val}_\infty (\alpha)}$ and put $l = \lfloor \mathbf{v}\rceil + 1$, where $\lfloor \cdot \rceil$ rounds to the nearest integer, rounding up in case of a tie. In the $i$-th iteration, given $(\beta_1, \ldots , \beta_{i-1})$ from previous iterations, we compute a minimum $\beta_i$ of the fractional $O_F$-ideal $\beta_{(i)}^{-2} O_F$ close to $\mathbf{v}/2^{l-i} - 2 \, \mbox{val}_\infty(\beta_{(i)})$, where $\beta_{(i)} = \prod_{j=1}^{i-1} \beta_j^{2^{i-j}}$. 


\begin{algorithm}   \caption{CompRep \cite[Algorithm 4.6]{EisentragerHallgren}}\label{alg:comprep}
\begin{algorithmic}[1]
\REQUIRE A principal ideal $A=\alpha O_F$ and the vector $\text{val}_\infty (\alpha)
\in \mathbb{Z}^r$    
\ENSURE A compact representation of $\alpha$
    
    \STATE $\mu \leftarrow $Reduce(A,$\vec{0}$)
    \STATE $l\leftarrow \lfloor \log_2 \norm{\text{val}_\infty (\mu)-\text{val}_\infty(\alpha)}_\infty \rceil +1$
    \STATE $B\leftarrow 1\cdot O_F$, $\beta_0 \leftarrow 1$, $v_\beta \leftarrow \vec{0}\in \mathbb{Z}^{r}$
    \FOR {$i\in \{1, \dots, l\}$} 
    \STATE $t\leftarrow \begin{bmatrix} \left\lfloor (v_{P_{\infty,1}}(\mu) - v_{P_{\infty,1}}(\alpha))/2^{l-i} \right\rceil  & \dots &  \left\lfloor (v_{P_{\infty,r}}(\mu) - v_{P_{\infty,r}}(\alpha))/2^{l-i} \right\rceil \end{bmatrix}^T $
    \STATE $B\leftarrow 1/\beta_{i-1}^2 B^2$
    \STATE $\beta_i \leftarrow $Reduce(B, $t-2 v_{\beta}$)
    \STATE $v_\beta \leftarrow 2v_\beta + \text{val}_{\infty}(\beta_i)$
    \ENDFOR
    \RETURN $(\mu, (\beta_1, ... \beta_l))$;
    \end{algorithmic}
\end{algorithm}

Given the multiplicative structure of compact representations, it is relatively straightforward to devise algorithms for computing compact representations of products, powers and norms of elements given in compact representation. It is also easy to find the value at any place $P \in P(F)$ of an element in compact representation, ascertain whether an element in compact representation belongs to $O_F$ (by checking that all its values at the finite places are non-negative), and determine whether two elements in compact representation are associate (by comparing their values at all the finite places in their support). We omit the details here and refer to \cite[Section 2.4.2]{Leem_thesis} for explicit descriptions of these algorithms.

\section{Solving norm equations}\label{sec:algdesc}

In this section, we first describe several techniques for solving norm equations, beginning with the only method found in previous literature, due to Ga\'{a}l and Pohst~\cite{GaalPohstnorm}. Then we present two new algorithms for accomplishing this task. The method in Section \ref{subsec:PGCR_description} improves on the exhaustive search approach taken by Ga\'{a}l-Pohst and incorporates compact representations. Section \ref{sec:PIGCRalgdesc} introduces a new algorithm that uses index calculus techniques and also makes use of compact representations.

\subsection{Ga\'{a}l-Pohst}

The idea of this method is to look for all non-associate solutions of \eqref{eq:ne_defn} in a certain region and check that each solution candidate has the correct norm. Since the search space is not explicitly given in \cite{GaalPohstnorm}, we describe it here. 

Suppose $\alpha$ is a solution of (\ref{eq:ne_defn}), given in standard representation with respect to a reduced basis $\mathcal{B} = \{ \omega_1, \ldots , \omega_n \}$ of $F/k(x)$.  Then
\begin{equation}\label{eq:alphaformredbasis}
    \alpha=\sum_{i=1}^n \lambda_i \omega_i,
\end{equation}
where $\lambda_i\in k[x]$ for $1\le i \le n$. Since $\mathcal{B}$ is a reduced basis, we have
\begin{equation*}\label{eq:max_norm_alpha_max_qdeglambdaomega}
    \infnorm{\alpha} = \max_{1\le i \le n}\{{\deg(\lambda_i)}+\infnorm{\omega_i}\},
\end{equation*}
so $\deg(\lambda_i) \le  \infnorm{\alpha} -\infnorm{\omega_i}$ for $1 \le i \le n$. Assume that $\alpha$ is minimal among its associate elements with respect to the maximum norm. Then an upper bound on $\infnorm{\alpha}$ produces a degree bound on $\lambda_i$ which yields a finite space that we can search for non-associate solutions of \eqref{eq:ne_defn} in $O_F$.

Let $\epsilon_1, \ldots, \epsilon_r$ be a system of fundamental units of $F$ and $P \in P_\infty(F)$. For any solution $\beta \in O_F$ of \eqref{eq:ne_defn}, there exist $x_1, \ldots , x_r \in \mathbb{R}$ such that
\begin{equation*}
v_P(\alpha) = \sum_{j=1}^{r} x_j v_P(\epsilon_j)+\frac{1}{n}v_P(c)\\
= \sum_{j=1}^{r} x_j v_P(\epsilon_j)-\frac{e_P}{n} \deg(c),
\end{equation*}
where $e_P$ is the ramification index of $P$. This identity is given in \cite[p.\ 244]{GaalPohstnorm} without proof, but can be derived by adapting the reasoning in \cite[Sections 5.3 and 6.4]{PohstZassenhaus} from number fields to function fields; for details, see \cite[Lemma 3.1]{Leem_thesis}.

Now put $\alpha = \beta \prod_{j=1}^r \epsilon_j ^{-\lfloor x_j\rceil}$. Then $\alpha$ is associate to $\beta$. Since $1/2 \ge a - \lfloor a \rceil \ge -1/2$ for all $a \in \mathbb{R}$, a simple calculation yields
\begin{equation} \label{eq:alphacoeff}
\theta_P - \frac{e_P}{n} \ge v_P(\alpha) \ge -\theta_P -\frac{e_P}{n}\deg(c) \quad \mbox{ where } \theta_P = \frac{1}{2} \sum_{j=1}^r | v_P(\epsilon_j)| .
\end{equation}
The lower bound implies 
\begin{equation} \label{eq:bound_for_max_norm_alpha}
\infnorm{\alpha} \le \Theta \quad \mbox{ where } \Theta=\max_{P \in P_\infty(F)} \left\{\frac{\theta_P}{2e_P} \right\}+\frac{1}{n}\deg(c),
\end{equation}
\begin{equation} \label{eq:degboundoflambdai}
    \deg \lambda_i \le \Theta - \infnorm{\omega_i} \quad \mbox{ for $1 \le i \le n$}. 
\end{equation}

We can invoke Algorithm \ref{alg:Sunitlat} to compute the values of a system of fundamental units at the infinite places of $F$ and compute the bounds given in \eqref{eq:degboundoflambdai} for $1 \le i \le n$. Then we compute the norm of every $\alpha$ of the form (\ref{eq:alphaformredbasis}) such that the coefficients~$\lambda_i$ satisfy \eqref{eq:degboundoflambdai} and only retain $\alpha$ if its norm is equal to $c$ up to a multiple in $k^\times$. Once all solutions have been found, we remove associate solutions via the procedure described at the end of Section~\ref{sec:FFandNE}. 

\subsection{Improved exhaustive search algorithm}\label{subsec:PGCR_description}

In this section, we describe a new exhaustive search algorithm for solving norm equations which makes use of compact representations. We assume $\deg P_{\infty, r+1} = 1$. Let
\[ S_{c,0} = \{ P \in P_0(F) \mid v_P(c) \ne 0 \} \quad \mbox{and} \quad S_c = S_{c,0} \cup P_\infty(F) \ . \]
Then every solution $ \alpha \in O_F$ of \eqref{eq:ne_defn} is an $S_c$-unit. So if we can bound the values $v_P(\alpha)$ for all $P \in S_c$, then we can search the region defined by these bounds for solutions. 

Bounds on $v_P(\alpha)$ for $P \in P_\infty(F)$ are given in \eqref{eq:alphacoeff}. Note that the quantities $\theta_P$ can easily be obtained from the unit value matrix $M_{P_\infty(F)}$ computed in Algorithm~\ref{alg:Svalmat}. To obtain bounds on $v_P(\alpha)$ for $P \in P_0(F)$, 
write \eqref{eq:ne_defn} in the form $\alpha\beta = c$ where $\beta 
= \textrm{Norm}_{F/k(x)}(\alpha)/\alpha \in O_F$, which implies  
$0 \le v_P(\alpha) 
\le v_P(c)$ for all $P \in P_0(F)$.

We use these bounds to form inputs for Algorithm \ref{alg:comprep} to compute compact representations of solution candidates of \eqref{eq:ne_defn}. Write 
\[ S_{c_0} = \{ P_1, \ldots , P_{|S_{c,0}|} \} \ , \qquad P_\infty(F) = \{ P_{\infty,1}, \ldots , P_{\infty, r+1} \}. \]
The solutions of \eqref{eq:ne_defn}, up to associates, are in one-to-one correspondence with the principal ideals $\alpha O_F$ dividing $cO_F$. For $1 \le i \le |S_{c,0}|$, let $\mathfrak{p}_i$ be the prime ideal corresponding to $P_i \in S_{c,0}$. Then all the integral $O_F$-ideals dividing $cO_F$ are of the form 
\begin{equation*}\label{eq:PGCRformingI}
I= \prod_{i=1}^{|S_{c,0}|} \mathfrak{p}_i ^{v_i},
\end{equation*}
where $0\le v_i\le v_{P_i}(c)$ for $1 \le i \le |S_{c,0}|$. If $I$ is principal, say $I = \alpha O_F$, then $v_P(\alpha)$ satisfies the bounds \eqref{eq:alphacoeff} for all $P \in P_\infty(F)$. An additional constraint is given by the fact that the principal divisor of $\alpha$ has degree zero. In other words, we only need to consider tuples $(v_1, \ldots , v_{|S_{c,0}|})$ and $(v_{\infty,1}, \ldots , w_{\infty, r+1})$ such that 
\begin{equation} \label{eq:finite-values}
    0 \le v_i \le v_{P_i}(c) \quad \mbox{for } 1 \le i \le |S_{c,0}|,
\end{equation}
\begin{equation} \label{eq:infinite-values}
    -\theta_{P_{\infty,i}} - \frac{e_{P_{\infty,i}}}{n} \deg(c) \le v_{\infty,i} \le \theta_{P_{\infty,i}} - \frac{e_{P_{\infty,i}}}{n} \deg(c)  \quad \mbox{for } 1 \le i \le r+1, 
\end{equation}
\begin{equation} \label{eq:PGCR_deg0div}
\sum_{i=1}^{|S_{c,0}|} v_i \deg Q_i + \sum_{i=1}^{r+1} v_{\infty,i} \deg P_{\infty,i} = 0 . 
\end{equation}
These conditions are necessary, but not sufficient, for $\alpha$ to be a solution of \eqref{eq:ne_defn}. Nevertheless, the constraint \eqref{eq:PGCR_deg0div} in particular significantly cuts down the number of compact representations that need to be computed.  

For every pair $(I, V_\infty)$, with $V_\infty = (v_{\infty_1}, \ldots , v_{\infty,  r+1})$, that satisfies these conditions, we compute a compact representation $\mathbf{t} =$ CompRep$(I, V_\infty)$. We then test that $\mathbf{t}$ represents an element in $O_F$ and that this element has the correct norm. 
We discard $\mathbf{t}$ if it represents an element that is associate to a solution already found. Algorithm \ref{alg:PGwithCR} shows the whole process.

\begin{algorithm}    \caption{Solving \eqref{eq:ne_defn} via improved exhaustive search}
\label{alg:PGwithCR}
    \begin{algorithmic}[1]

\REQUIRE $c \in k[x]\setminus \{0\}$, a reduced basis $\mathcal{B} = \{\omega_i \vert 1\le i \le n\}$ of $O_F$, a unit value matrix $M_{{P}_\infty(F)} = (m_{i,j})$ 
\ENSURE A set $\mathcal{R}$ of all non-associate solutions of the equation (\ref{eq:ne_defn}) in compact representation
    \STATE $\mathcal{R}\leftarrow \emptyset$
    \STATE $S_{c,0}\leftarrow \{P\in{P}_0(F) \vert v_P(c)\ne 0 \}$    
    \STATE $ v_l\leftarrow \begin{bmatrix} \displaystyle \frac{1}{2} \sum_{i=1}^{r} \vert m_{i,1} \vert &  \dots & \displaystyle \frac{1}{2} \sum_{i=1}^{r} \vert m_{i,r+1} \vert\end{bmatrix}$ \medskip
    \STATE $v_c 
    \leftarrow \begin{bmatrix}\displaystyle \frac{e_{P_{\infty, 1}}}{n}\deg(c) & \dots & \displaystyle\frac{e_{P_{\infty, r+1}}}{n}\deg(c) \end{bmatrix}$ \medskip
    \FOR {$(v_1, \dots, v_{\vert S_{c,0}\vert })$  where $0\le v_i\le v_{P_i}(c)$} 
    \FOR {$V_\infty= ( v_{\infty, 1}, \dots, v_{\infty, r+1})$ where $-(v_l)_i -(v_c)_i \le v_{\infty, i} \le (v_l)_i - (v_c)_i$}
    \IF {$\displaystyle \sum_{i=1}^{\vert S_{c,0}\vert} v_i \deg P_i + \sum_{j=1}^{r+1} v_{\infty, j } \deg P_{\infty, j} = 0$}
    \STATE $\mathbf{t} \leftarrow $ CompRep$(I, V_\infty)$
    \IF {$\mathbf{t} \in O_F$ and $\mathbf{t}$ represents an element of norm $\zeta c$ with $\zeta \in k^\times$}
    \IF {$\mathbf{t}$ represents an element that is is not associate to $\alpha$ for any $\mathbf{t}_\alpha \in \mathcal{R}$}
    \STATE $\mathcal{R}\leftarrow \mathcal{R}\cup \{\mathbf{t}\}$
    \STATE \algorithmicbreak   $~V_\infty$
    \ENDIF
    \ENDIF
    \ENDIF
    \ENDFOR
    \ENDFOR
    \RETURN $\mathcal{R}$    \end{algorithmic}
\end{algorithm}

\subsection{Index calculus}\label{sec:PIGCRalgdesc}

In this section, we describe a new exhaustive search algorithm for solving norm equations which also makes use of compact representations, so we assume again that $\deg P_{\infty, r+1} = 1$.  Unlike the previous exhaustive search techniques, which enumerate all elements within a large search region, this algorithm enumerates ideals $I$ that divide $cO_F$ and conducts principal ideal tests by solving matrix equations involving a precomputed $S$-unit value matrix. Using the solutions of the matrix equations, we compute compact representations of solutions of \eqref{eq:ne_defn}.
%
%

The solutions $\alpha$ of \eqref{eq:ne_defn}, up to associates, are in bijection with the principal ideals $\alpha O_F$ of norm $c\zeta$ with $\zeta \in k^\times$. By \eqref{eq:finite-values}, any such ideal must necessarily divide $c O_F$. So in order to find all solutions, it suffices to consider $O_F$-ideals $I$ that divide $c O_F$. If $I$ is principal and has the correct norm, then a generator of $I$ is a solution of \eqref{eq:ne_defn}. 
%


Let $S_c$ and $S_{c,0}$ as in Section~\ref{subsec:PGCR_description}. In order to enumerate all ideals $I$ that divide~$cO_F$, we factor $cO_F$ as 
\begin{equation*} 
c O_F = \prod_{i=1}^{|S_{c,0}|} \mathfrak{p}_i^{v_{P_i}(c)},  
\end{equation*}
where for each $i$, $\mathfrak{p}_i$ is the $O_F$-prime ideal corresponding to the place $P_{i} \in S_{c,0}$. Then we perform principal ideal tests on all ideals $I$ dividing $cO_F$, which are precisely of the form 
$$I=\prod_{i=1}^{\vert S_{c,0}\vert}\mathfrak{p}_i ^{v_{P_i}(I)},$$
such that $0 \le v_{P_i}(I) \le v_{P_i}(c)$ for all $1\le i \le \vert S_{c,0}\vert$, where $v_{P_i}(I) = v_{\mathfrak{p}_i}(I)$. 
%
%

There is a principal ideal test, implemented in Magma, which is an index calculus algorithm that uses Hess's randomized relation search algorithm \cite[Algorithm 5.5]{hessphd}. This algorithm finds a factorization of an ideal equivalent to $I$ by searching relations. When $I$ is principal, the algorithm returns a generator in ``factored form". The factored form has subexponentially many terms, each of which has subexponential size in the size of inputs. In our context, the prime ideal factorization of $I$ is already known, and we wish to compute a compact representation of a generator of $I$ if $I$ is principal. Thus, instead of using the existing algorithm, we solve a matrix equation for each $I$ to determine whether or not $I$ is principal and to derive inputs for computing a compact representation. 


Let $\{\epsilon_1, \epsilon_2, \dots, \epsilon_{|S_c|-1}\}$ be a system of fundamental $S_c$-units. 
Every solution $\alpha$ of \eqref{eq:ne_defn} is an $S_c$-unit, so there exist integers $x_i \in \mathbb{Z}$ such that
\begin{equation}  \label{eq:alpha_as_Sunit}
\alpha= \prod_{i=1}^{|S_{c}|-1} \epsilon_{i}^{x_i} .
\end{equation}
%
We form a matrix $M_{S_c,0}$ from the columns of the $S_c$-value matrix $M_{S_c}$ that correspond to the places in $S_{c,0}$. Here, $M_{S_c}= $ SValMat($S_{c}$) is precomputed using Algorithm \ref{alg:Svalmat}.

Now consider the matrix equation
\begin{equation}\label{eq:matrixeq_NE3princ}
\begin{bmatrix}
v_{P_1}(\epsilon_1) & v_{P_1}(\epsilon_2) & \dots & v_{P_1}(\epsilon_{\vert {S_c}\vert -1})\\ 
v_{P_2}(\epsilon_1) & v_{P_2}(\epsilon_2) & \dots & v_{P_2}(\epsilon_{\vert {S_c}\vert -1})\\ 
 \vdots & \vdots & \ddots & \vdots \\ 
v_{P_{\vert S_{c,0}\vert}}(\epsilon_1) & v_{P_{\vert S_{c,0}\vert}}(\epsilon_2) & \dots & v_{P_{\vert S_{c,0}\vert}}(\epsilon_{\vert {S_c}\vert -1}) \end{bmatrix}\begin{bmatrix}
x_1\\ 
x_2\\ 
\vdots\\ 
x_{\vert {S_c}\vert -1}
\end{bmatrix}= 
\begin{bmatrix}
v_{P_1}(I)\\ 
v_{P_2}(I)\\ 
\vdots\\ 
v_{P_{\vert S_{c,0}\vert}}(I)
\end{bmatrix}.
\end{equation}
It is easy to verify that \eqref{eq:matrixeq_NE3princ} has a solution $\begin{bmatrix} x_1 & \dots & x_{\vert {S_c}\vert -1} \end{bmatrix}^T$  if and only if $I$ is principal, with a generator given by \eqref{eq:alpha_as_Sunit}.
%
%
%
Any such solution gives rise to infinitely many solutions 
as $M_{S_c,0}$ has fewer rows than columns. 
However, any two solutions of \eqref{eq:matrixeq_NE3princ} correspond to associate solutions of \eqref{eq:ne_defn}. So for any $O_F$-ideal~$I$ dividing~$cO_F$, we need only find one solution of \eqref{eq:matrixeq_NE3princ}, and this should be the one that gives rise a solution $\alpha_0$ of \eqref{eq:ne_defn} whose norm $\infnorm{\alpha_0}$ is small or even minimal. We proceed as follows.

First, choose a solution $\mathcal{X}=\begin{bmatrix} x_1 & \dots & x_{\vert S_{c}\vert-1} \end{bmatrix}^T$ of \eqref{eq:matrixeq_NE3princ}
that corresponds to $\alpha =\prod_{i=1}^{\vert {S_c}\vert -1} \epsilon_i^{x_{i}}$. Then compute the vector $v_{\alpha, \infty}$ of the values of $\alpha$ at the infinite places by 
\begin{equation*} 
    v_{\alpha, \infty} = M_{S_c,\infty}^T \mathcal{X},
\end{equation*}
where $M_{S_c,\infty}$ is the matrix consisting of the columns of $M_{S_c}$ that are not in $M_{S_c,0}$, i.e.\ correspond to the infinite places of $F$. Note that even if $\mathcal{X}$ is short in the Euclidean norm, $\infnorm{v_{\alpha, \infty}}$  could still be very large. So 
%
we compute a vector $v_0$ in the unit lattice that is close to $v_{\alpha, \infty}$ with respect to the Euclidean norm and form the vector 
\begin{equation*} 
    v_{\alpha_0, \infty}=v_{\alpha, \infty}-v_0.
\end{equation*}
This vector is short in the Euclidean norm and thus corresponds to a generator $\alpha_0$ of $I$ such that $\infnorm{\alpha_0}$ is small. From $I$ and $v_{\alpha_0, \infty}$, we can compute a compact representation $\mathbf{t}=$ CompRep($I$, $v_{\alpha_0, \infty}$) of $\alpha_0$. If it has norm $c$, it represents a solution of~(\ref{eq:ne_defn}). 

\begin{algorithm} 
    \caption{Solving norm equations via index calculus}
    \label{alg:index_calculus}
\begin{algorithmic}[1]\label{alg:princwithCR}
    \REQUIRE $c \in k[x]\setminus k$, , the maximal order $O_F$ of $F$, an $S_c$-unit value matrix $M_{S_c}$
    \ENSURE A set $\mathcal{R}$ of all non-associate solutions of (\ref{eq:ne_defn}) that are in $O_F$ in compact representation.
    \STATE $\mathcal{R}\leftarrow \emptyset$
    \STATE $S_{c,0}\leftarrow \{P\in {P}_0(F) \vert v_P(c)\ne 0 \}$
    \STATE $M_{S_c,0} \leftarrow$ the matrix of the columns of $M_{S_c}$ corresponding to the places in $S_{c,0}$
    \STATE $M_{S_c,\infty} \leftarrow$ the matrix of the columns of $M_{S_c}$ corresponding to the places in ${P}_\infty(F)$
    \STATE $M_{{P}_\infty(F)}
    \leftarrow$SValMat(${P}_\infty(F)$)
    \FOR {every $I \vert c O_F$ such that $I=\prod_{i=1}^{\vert S_c \vert } \mathfrak{p}^{v_{P_i}(I)}$}
    \IF {$\textrm{Norm}_{F/k(x)}(I)$/$c \in  k  ^\times$,}
    \IF {\eqref{eq:matrixeq_NE3princ} is consistent,}
    \STATE $\mathcal{X} \leftarrow$ a solution of \eqref{eq:matrixeq_NE3princ} that is short in the Euclidean norm
    \STATE $v_0\leftarrow$ a vector in the lattice generated by the rows of $M_{{P}_\infty(F)}$, that is closest to $M_{S_c,\infty}^T \mathcal{X}$
    \STATE $v \leftarrow M_{S_c,\infty}^T \mathcal{X} -v_0$
    \STATE $\mathbf{t} \leftarrow$ CompRep$(I, v)$
        \STATE $\mathcal{R}\leftarrow \mathcal{R}$ $\cup$ $\{\mathbf{t}\}$
    \ENDIF    
    \ENDIF\ENDFOR
    \RETURN $\mathcal{R}$
    \end{algorithmic}
\end{algorithm}

In Example \ref{ex:all},
we compute the search space of each algorithm for the same norm equation and compare the run times. The computation was performed on an Intel Xeon CPU E7-8891 v4 with 80 64-bit cores at 2.80GHz with the help of Magma.

\begin{exmp}\label{ex:all}
Let $F/\mathbb{F}_5(x)$ be an extension of degree $n=3$ defined by a root of
\[f(t) = t^3 + (4x^3 + 3x^2 + 1)t^2 + (3x^3 + 4x^2 + 4x + 2)t + 2x^3 + x.\]
Then 
$F/k(x)$ has two infinite places $P_{\infty,1}$ and $ P_{\infty, 2}$ with ramification indices $e_{P_{\infty, 1}}=e_{P_{\infty, 2}}=1$,
and hence unit rank $r=1$.

Let $c=x+4$. The prime ideal factorization of $cO_F$ is $cO_F=\mathfrak{p}_1\mathfrak{p}_2$, where these two prime ideals correspond to two finite places $P_1$ and $P_2$ with $v_{P_1}(c)=v_{P_2}(c)~=~1$. 

The search space of Ga\'{a}l-Pohst is determined by the degree bounds in \eqref{eq:degboundoflambdai}. In this example, we have
%
$\deg \lambda_1 \le 347+\frac{1}{3} $, $\deg \lambda_2 \le 344+\frac{1}{3} $, and $ \deg \lambda_3 \le 344+\frac{1}{3}.$ 
%
Thus, a search for solutions requires computing the norms of $5^{348+345+345}=5^{1038}$ elements in $O_F$. 

The search space of Algorithm \ref{alg:PGwithCR} is the number of compact representations computed in Step 8, which is equivalent to the number of tuples satisfying \eqref{eq:finite-values}, \eqref{eq:infinite-values}, and \eqref{eq:PGCR_deg0div}. In our case, we need to find tuples $(v_1, v_2, v_{\infty, 1}, v_{\infty, 2})$. To satisfy the degree bound in \eqref{eq:PGCR_deg0div}, we only need to choose the first 3 numbers to determine the tuple. Since we have $0\le v_1, v_2\le 1$ and $-347-\frac{1}{3}\le v_{\infty_1}\le 347-\frac{1}{3}$, we have up to $2\cdot2\cdot(347\cdot2+1) = 2980$ possible tuples which is much less than the search bound for Ga\'{a}l-Pohst. 

Lastly, the number of ideals to enumerate in Algorithm \ref{alg:princwithCR} is the number of pairs $(v_1, v_2)$. With the same bounds on $v_1, v_2$ above, we only need to search 4 ideals which is significantly less than the previous two algorithms.

The search of Ga\'{a}l-Pohst did not finish within 4 days, so we terminated the computation. Our improved exhaustive search algorithm only took 114.830 CPU seconds, and the index calculus algorithm only took 0.180 CPU seconds for the entire process.
\end{exmp}

\section{Complexity Analysis}

In this section, we analyze the complexity of the compact representation and norm equation algorithms. 
%
Throughout, $F/k(x)$ is represented by a monic irreducible polynomial $f(t) = t^n + a_{n-1} t^{n-1} + \cdots + a_0 \in k[x][t]$. The size of this representation is captured by the quantity 
\begin{equation}\label{eq:definition_Cf}
    C_f = \max \left \{ \left \lceil \frac{\deg a_i (x)}{i} \right \rceil \bigg\vert~~ 1\le i \le n\right \}.
\end{equation}
Note that $C_f = \bigO(g)$ when $g \rightarrow \infty$; see \cite[Corollary 3.5]{Jens_Genus}. Except for the original Ga{\'a}l-Pohst method, we assume that $F/k(x)$ has a place of degree one. 

Our asymptotic run times count bit operations and are expressed as functions of $q = |k|$, $n = [F:k(x)]$, $g$ (the genus of $F$), and the sizes of the inputs specific to each algorithms. Some complexity estimates include other quantities, such as the regulator $R_F$ or the unit rank $r$ of $F/k(x)$. If all these quantities are present, $O()$ constants should be understood as true constants. Later, we will consider asymptotics where one of $q$, $n$, $g$ and $\deg(c)$ (in the case of solving norm equations~\eqref{eq:ne_defn}) grows and the others are assumed to be fixed; these $O()$ constants will then depend on the fixed quantities.  For any quantity $X$, we simplify any power of $\log X$ to writing $X^{\td}$.


For basic arithmetic ingredients, we assume the following complexities:

\begin{itemize}
    \item Multiplication of two elements in $k$: $O((\log q)^{1+\td}) = O(q^{\td})$ \cite{poly_mult};
    \item Multiplication of two polynomials in $k[x]$ of degree $d$: $\bigO(d^{1+\td} q^{\td} )$ \cite{poly_mult};
    \item Computing the determinant of a matrix $M = (m_{ij}) \in k[x]^{n \times n}$:  $\bigO(n^{\omega+\td} q^{\td} \lceil s(M) \rceil )$, where $\omega < 2.37286$ (see \cite[Proposition 3.3]{comp_detandhnf} and \cite{AlmanWilliams_matmult}) and 
    \[ s(M) = \frac{1}{n} \sum_{i=1}^{n} \left(\max_{1\le j\le n}\deg(M)_{i,j} \right)\]
    is the average column degree of $M$.
    \item Factoring a polynomial of degree $d$ in $k[x]$: $\bigO(d^\omega q^\td)$ using Berlekamp's algorithm \cite[Theorem 14.32]{vonzurGathen}. 
\end{itemize}




Sizes of elements are measured in \emph{heights} which are defined as follows:

\begin{itemize} \itemsep 4pt
    \item For $\lambda=\lambda_1/\lambda_2 \in k(x)$ with coprime polynomials $\lambda_1, \lambda_2 \in k[x]$, $\lambda_2\ne 0$, we define $\hh(\lambda) = \max\{\deg(\lambda_1) , \deg(\lambda_2)\}$.
    \item For a $k(x)$-basis $\mathcal{B}= \{\omega_1, \ldots , \omega_n \}$ of $F$ and $\alpha = \sum_{i=1}^n \lambda_i \omega_i \in F$ with $\lambda_i \in k(x)$ for $1\le i \le n$, we define $\hh_{\mathcal{B}}(\alpha) = \max_{1\le i\le n}\hh(\lambda_i).$ \\
    For a precomputed fixed reduced basis $\mathcal{B}$, 
    we write $\hh(\alpha)$ for $\hh_{\mathcal{B}}(\alpha)$.
    \item For a divisor $D = \sum_P n_P P$ of $F$, we define $\hh(D) =   \sum_P |n_P| \deg P $.
\end{itemize}
We assume that we have precomputed a reduced basis $\mathcal{B}=\{\omega_1, \ldots , \omega_n\}$ of $F/k(x)$ and polynomials $a_{ijm} \in k[x]$ such that $\omega_i \omega_j = \sum_{m=1}^n a_{ijm} \omega_m$ for $1 \le i, j \le n$. The elements of a reduced basis are short; specifically
\begin{equation}\label{eq:bound_infinite_norm_omega_n}
     \infnorm{\omega_n} \le \left\lceil \frac{2g-1}{n} \right\rceil + 1
\end{equation}
by \cite[Theorem 5.4.1]{ATthesis}. 

Fractional $O_F$-ideals $I$ are given in \emph{Hermite Normal Form} (HNF) representation, i.e.\ as a pair $(M_I, d(I))$. Here, $d(I)$ is the denominator of $I$, i.e.\ the monic polynomial $d \in k[x]$ of minimal degree such that $dI \subseteq O_F$, and $M_I$ is the coefficient matrix of a $k[x]$-basis of $d(I)I$ in HNF.

The next two lemmas provide the cost of norm computation.

\begin{lem}\label{lem:comp_norm_of_an_element}
For $\alpha \in F$, computing $\textrm{Norm}_{F/k(x)}(\alpha)$ requires $\bigO (n^3 d_\alpha^{1+\td}q^{\td})$ bit operations, where $d_\alpha=\max  \{\hh
(\alpha), 2\left( \left\lceil \frac{2g-1}{n} \right\rceil+1 \right)\}$.
\begin{proof}
We have $\textrm{Norm}_{F/k(x)}(\alpha) = \det(M_\alpha)$, where $M_\alpha \in k(x)^{n\times n}$ is the unique matrix such that
$\alpha \, [\omega_1 , \dots , \omega_n] = [\omega_1 , \dots , \omega_n] \, M_{\alpha}$. Writing $\alpha = \sum_{i=1}^n \lambda_i \omega_i$ and $M_\alpha = (m_{ij})$, we have $m_{jl} = \sum_{i=1}^n \lambda_i a_{ijl}$ and $\deg m_{jl} \le 2d_\alpha$ by \eqref{eq:bound_infinite_norm_omega_n}. So computing~$M_\alpha$ takes $\bigO\left(n^3 d_\alpha^{1+\td} q^{\td}\right)$ bit operations, and this dominates the cost of computing $\det M_\alpha$.
%
\end{proof}
\end{lem}


\begin{lem}\label{lem:norm_I_hnf}
For a a fractional $O_F$-ideal $I$ in HNF representation, computing $\textrm{Norm}(I)$ requires
    $\bigO\left( \left( n^{2+\td}\deg(\textrm{Norm}(d(I)I)) + (n\deg d(I))^{1+\td}\right)q^{\td}\right)$
bit operations.
\begin{proof}
Let $(M_I, d(I))$ be the HNF representation of $I$. Then $\textrm{Norm}(I) = \det M_I/d(I)^n$. By \cite[Proposition 5.1.17]{ATthesis}, we have $\infnorm{M_I} \le \deg(\text{Norm}(d(I)I))$. So the cost of computing $\det M_I$ is 
$\bigO(n^{2+\td}\deg(\text{Norm}(d(I)I))q^{\td})$, and that of computing $d(I)^n$ is $\bigO((n\deg(d(I)))^{1+\td}q^{\td})$.
\end{proof}
\end{lem}

\subsection{Compact representation}
The cost of computing a compact representations using Algorithm \ref{alg:comprep} is dominated by the calls to Algorithm \ref{alg:reduce} whose cost in turn is dominated by computing $k$-bases of at most $g+1$ Riemann-Roch spaces. We assume that $\deg P_{\infty,r+1} = 1$.
%
%
Let $D$ be a divisor of $F$. 
For brevity, we denote the cost of computing a $k$-basis of the Riemann-Roch space $L(D)$ by $\rr{\hh(D)}$. By \cite[Theorem 4.13]{Jens_Latt}, we have
\begin{equation}\label{eq:comp_RRcomputation}
\rr{\hh(D)} = \bigO \left(\left(n^5(\hh(D)+n^2C_f)^2+ n^{5+\td}C_f^{2+\td}\right) q^{\td}\right)
\end{equation}
bit operations, with $C_f$ as in (\ref{eq:definition_Cf}), and 
$\rr{m \hh(D)} = \rr{\hh(D)}$ for 
$m\in \mathbb{R}$.

\begin{lem} \label{lem:comp_reduce_withprincipalideal}
    Let $I$ be a fractional ideal and $v \in \mathbb{Z}^r$. On input $I$ and $v$, Algorithm~\ref{alg:reduce} requires $\bigO (g\,\rr{n\hh(\text{Norm}(I))+n\infnorm{v}+g})$ bit operations.
\begin{proof}
For the divisor $D$ formed in step 1, we have $\hh(D) \le \hh(\dvs(I)) + n \infnorm{v}$, where
\[ \hh(\dvs(I)) \le \sum_{P \in P_0(F)} |v_P(\dvs(I))| \deg P 
    \le \sum_{P \in P_0(F)} |v_P(\text{Norm}(I))| \deg P \le 2n\hh(\text{Norm}(I)) . \]
The last inequality can be obtained from the factorization of $\textrm{Norm}(I)$ into irreducible polynomials in $k[x]$; see \cite[Lemma 2.35]{Leem_thesis}. The interval containing $\ell$ forces $|\ell| \le h(D)+g$, so 
\[ \hh(D+\ell P_{r+1}) \le 2h(D) + g \le 2n\hh(\text{Norm}(I)) + n \infnorm{v} + g. \]
The loop in step 2 is executed $g+1$ times, so the result follows from \eqref{eq:comp_RRcomputation}.
\end{proof}
\end{lem}


\begin{lem}\label{lem:comp_of_comprep} Let $A$ be a principal $O_F$-ideal in HNF-representation, generated by an element $\alpha\in F$. On input $A$ and and $\textnormal{val}_\infty(\alpha)$,
Algorithm \ref{alg:comprep} requires
\begin{align*}
\bigO\left(g  \left(\rr{n\hh(\textrm{Norm}_{F/k(x)}(\alpha)) + g}+  \log\left(\norm{\textnormal{val}_\infty(\alpha)}_\infty+g\right) \rr{n^2+ng}\right)\right)
\end{align*}
bit operations.
\begin{proof}
We use the fact that $\beta = \mu/\alpha$ is a minimum of $O_F$, so $B = (\beta^{-1})O_F$ is a reduced $O_F$-ideal, and hence  $0 \le \deg (\dvs(B)) \le g$. By Lemma \ref{lem:comp_reduce_withprincipalideal}, the cost of step 1 is 
$\bigO\left(g\,\rr{n\hh(\textrm{Norm}_{F/k(x)}(\alpha)) + g}\right)$ bit operations. Similar reasoning shows that the cost of computing each $\beta_i$ in step is $\bigO \left(g\,\rr{n^2+ng} \right)$ bit operations. 
The number $l$ of loop iterations defined in step 2 can be bounded by
\begin{equation*} 
    l= \lfloor \log_2 (\norm{\text{val}_\infty(\beta)}\rceil + 1 \le \lfloor \log_2 (\norm{\text{val}_\infty (\alpha)}_\infty + g)\rceil +1. \qedhere 
\end{equation*}
\end{proof}
\end{lem}

Let $D$ be any divisor of $F$ and $\alpha \in L(D)$. Then  \cite[Lem.\ 3.5]{JHS_divisor_class} 
implies that 
%
\begin{equation}\label{lem:height_of_RRspace_elt}
    \hh(\alpha)= \bigO(\hh(D)+n).
\end{equation}
We can now bound the heights of the quantities comprising a compact representation. 

\begin{lem}\label{lem:height_CR}
Let $\mathbf{t}_\alpha = (\mu, \beta_{1},\dots,\beta_{l})=$ CompRep$(\alpha O_F$, $v_\infty(\alpha))$ be a compact representation of $\alpha \in F$, and let $C_f$ be as defined in (\ref{eq:definition_Cf}). Then the following hold.
\begin{align*}
l&=\bigO(\log\infnorm{v_\infty(\alpha)}+g),\\
\hh(\mu) & = 
\bigO(n\hh(\textrm{Norm}_{F/k(x)}(\alpha))+g+ n), \\
\hh(\beta_i) & = \bigO(n^2+ ng) ~\text{ for } 1\le i \le l.
\end{align*}

\begin{proof}
The bound on $l$ was established in he proof of Lemma \ref{lem:comp_of_comprep}. We have $\mu \in L(D)$ where $D = \dvs(\alpha O_F) + \ell P_{r+1}$. The bound on $\hh(\mu)$ follows from the bound on $\hh(D)$ given in the proof of Lemma \ref{lem:comp_of_comprep} and \eqref{lem:height_of_RRspace_elt}. 
%

By \cite[Proof of Prop.\ 4.11]{EisentragerHallgren}, each $\beta_i$ is a minimum in a fractional $O_F$-ideal~$B_i^2$ close to a vector $t_i = (t_{i1}, \ldots , t_{ir})$, where $B_i$ is a reduced ideal and 
$\norm{t_i}_\infty = O(n+g)$.
So $\beta_i \in L(D_i)$ where \[\hh(D_i)= \hh(\dvs(B_i^2))+\hh\left(\sum_{j=1}^r t_{ij} P_j + \ell P_{r+1}\right)= \bigO(g+n\norm{t_i}_\infty) =\bigO(n^2+ng).\]  
Thus, $\hh(\beta_i) = \bigO(n^2+ng)$ by \eqref{lem:height_of_RRspace_elt}.
\end{proof}
\end{lem}

\subsection{Ga\'{a}l-Pohst}\label{sec:PG_companalysis}\label{subsec:comp_PG}

The cost of the Ga\'{a}l-Pohst method \cite{GaalPohstnorm} is dominated by computing the norms of elements in the search space and checking whether solutions are associate. By \eqref{eq:degboundoflambdai}, the number of elements in the search space is bounded by
%
\begin{equation}\label{lem:num_of_elts_to_test}
\prod_{i=1}^{n} q^{\lfloor \Theta- \norm{\omega_i}_\infty \rfloor+1} 
\le q^{n \Theta}, 
\end{equation}
with $\Theta$ given by \eqref{eq:bound_for_max_norm_alpha}. 
%
%
%

%
To test associateness of two elements, we factor the principal ideals they generate, which is accomplished by factoring their norms using Berlekamp's algorithm. 
%
This yields the following complexity for Ga\'{a}l-Pohst's exhaustive search method.

\begin{thm}\label{thm:PG_complexity}
Let $T=r2^{r(r-1)/4 -1}R_F + \frac{1}{n}\deg c$, where $R_F$ is the regulator and $r$ the unit ran of $F/k(x)$, Let $d_T = \max \{T, \lceil\frac{2g-1}{n}\rceil\}$. Then the Ga\'{a}l-Pohst method can solve Equation (\ref{eq:ne_defn}) in
\[
\bigO\left(2^{nTq^{\td}}q^{\td} (n^3 d_T^{1+\epsilon}+(\deg c)^\omega ) \right)
\]
bit operations, when $n$, $g$, $q$ and $\deg c \rightarrow \infty$.
\begin{proof}
By \eqref{eq:degboundoflambdai}, we have $\hh
(\alpha) \le \Theta$ for every $\alpha$ in the search space.
By Corollary~\ref{cor:bound_LLLbasis_of_Unit_lattice}, we have $\Theta \le T$. 
Thus, computing the norm of each $\alpha$ can be done in time
$\bigO(n^3 d_T^{1+\epsilon}q^{\td})$ by Lemma \ref{lem:comp_norm_of_an_element}. By \eqref{lem:num_of_elts_to_test}, we compute the norms of up to $q^{n\Theta}$
elements. The cost of testing whether two solutions of \eqref{eq:ne_defn} are associate is $\bigO((\deg c)^\omega q^{\td})$ via norm factorization, and the number of tests that need to be performed is bounded above by $q^{n \Theta}$.
%
\end{proof}
\end{thm}

We briefly discuss the asymptotic complexity and the sizes of the solutions $\alpha$ produced by the Ga\'{a}l-Pohst method in the different asymptotic settings where one of $n$,  $g$, $\deg (c)$, $q$ tends to infinity and the other three quanities are fixed. 
By \eqref{eq:bound_for_max_norm_alpha}, we have $\hh
(\alpha) \le \Theta \le T.$
%
We note that $r \le n$ and use \eqref{ineq:RS'boundq^2g} to bound $R_F$.

\begin{itemize}
    \item $n \rightarrow \infty$: run time $2^{\bigO(n^2 2^{n^2/4})}$, $\hh
    (\alpha) = \bigO(n2^{n^2/4})$;
    \item $g \rightarrow \infty$: run time $2^{\bigO(q^g)}$, $\hh
    (\alpha) = 2^{\bigO(g)}$;
    \item $\deg c \rightarrow \infty$: run time $\bigO\left(2^{q^{\td}\deg c } (\deg c)^{\omega}\right)$, $\hh
    (\alpha) = O(\deg c)$;
    \item $q \rightarrow \infty$: run time $q^{\bigO(q^g)}$, $\hh
    (\alpha) = O(q^g)$.
\end{itemize}

\subsection{Improved exhaustive search}\label{subsec:comp_PGCR}

The cost of Algorithm \ref{alg:PGwithCR} is dominated by computing compact representations and their norms.
We assume that $F$ has at least one infinite place of degree $1$. 


The number of compact representations computed in Algorithm \ref{alg:PGwithCR} is bounded by  the number of tuples $(v_1, \ldots , v_{|S_{c,0}|}, v_{\infty,1}, \ldots , w_{\infty, r+1})$ that satisfy \eqref{eq:finite-values} and \eqref{eq:infinite-values}. The number of  $v_1, \ldots , v_{|S_{c,0}|}$ satisfying \eqref{eq:finite-values} is
%
\begin{equation}\label{eq:number_of_I}
 \prod_{i=1}^{\vert S_{c,0}\vert}( v_{P_i}(c) +1).
\end{equation}
Similarly counting the number of $( v_{\infty, 1} , \dots , v_{\infty,r+1})$ that satisfy \eqref{eq:infinite-values}, we obtain an upper bound of
%
\begin{equation}\label{eq:comp_pgcr_numberofCR}
\prod_{i=1}^{\vert S_{c,0}\vert} (v_P(c)+1)\prod_{j=1}^{r}(2\theta_{\infty,j} +1)
\end{equation}
on the number of compact representations computed in Algorithm \ref{alg:PGwithCR}. 
This yields the following cost estimate for Algorithm \ref{alg:PGwithCR}.

\begin{thm}\label{lem:comp_PGCR}
    With 
    $T'=r2^{r(r-1)/4 -1}R_F$, Algorithm \ref{alg:PGwithCR} can solve Equation (\ref{eq:ne_defn}) in 
    \begin{equation*}
    \begin{split}
    \bigO\Biggl(\bigl(g \rr{n\deg (c) + g} + g \log (T'+g)\rr{n^2+ng}\\+n^{5+\td}(T' \deg c)^{1+\td}+(\deg c)^\omega \bigr)q^{\td}(2T' +1)^r  \left(\prod_{P\in  S_{c,0}} (v_P(c)+1)\right)\Biggr)    \end{split}
    \end{equation*}
    bit operations, when $n$, $g$, $q$ and $\deg c \rightarrow \infty$.
    \begin{proof}
    The number of compact representations computed in Algorithm \ref{alg:PGwithCR} is given in \eqref{eq:comp_pgcr_numberofCR}.
    By Lemma \ref{lem:comp_of_comprep}, computing each compact representation takes
    \[\bigO\left(g q^{\td}\left( \rr{n\deg (c) + g} + \log (\max_i \theta_{\infty,i} +g)\rr{n^2+ng}\right)\right)\]
    bit operations.
    By Corollary \ref{cor:bound_LLLbasis_of_Unit_lattice}, we have $\theta_{\infty,i} \le \Theta \le T',$
    for $1\le i \le r+1$. The cost of testing whether or not such compact representation is in $O_F$ is 
    $\bigO((\deg c)^\omega q^{\td})$. For each compact representation computed in step 8, we have $l=  \bigO\left(\log\norm{M_{S_c,\infty}^T}_\infty \right)$. From Lemma \ref{lem:height_CR}, we also have $\hh(\beta_i) = \bigO(n^2+ng)$ for $1\le i \le l $ and $\hh(\mu) = \bigO(n \hh(Norm(I)) + g+n = \bigO(n \deg c + g+n)$ because $I$ divides $cO_F$. 
    By \cite[Lem. 2.48]{Leem_thesis}, the cost of computing the norm of such a compact representation 
    is thus $\bigO(n^{5+\td}(T' \deg c)^{1+\td}q^\td)$.  Finally, the cost of testing associateness of any two such compact representations is again the same as that of factoring $cO_F$, i.e.\
    $\bigO((\deg c)^\omega q^{\td})$.
    \end{proof}
\end{thm}

Again we analyze the complexity of Algorithm \ref{alg:PGwithCR} under different asymptotic assumptions, with one of $n$,  $g$, $\deg c$, $q$ tending to infinity and the others remaining fixed. 
To simplify the expression in Theorem \ref{lem:comp_PGCR}, we bound \eqref{eq:number_of_I}.
This quantity varies greatly depending on the factorization of $c$. It takes on its minimal possible value when $cO_F$ has only one unramified (i.e.\ inert) place $P$, in which case $v_P(c)~=~1$. Its maximum occurs when $c$ splits into linear factors 
and each linear factor splits completely, in which case 
$v_P(c)=1$ for all $P\in S_{c,0}$ and $\vert S_{c,0}\vert = n\deg c$. Thus, 
\begin{equation}\label{eq:boundon_prod_vpc+1}
    2\le \prod_{P\in S_{c,0}}(v_P(c)+1) \le 2^{n\deg c}.
\end{equation}
Along with these bounds, we again use $r \le n$  and bound $R_F$ via \eqref{ineq:RS'boundq^2g}.

\subsubsection{Case $n \rightarrow \infty$}
Here, $T'=
\bigO(n2^{n^2/4})$. Using the upper bound in \eqref{eq:boundon_prod_vpc+1}, we obtain 
 $\vert S_c \vert =  \bigO(n)$. 
If $cO_F$ is inert, in which case the lower bound in \eqref{eq:boundon_prod_vpc+1} applies, the asymptotic run time improves by a factor of $2^n$.
In this case, $\hh(\dvs(cO_F))=\bigO(n)$. By Proposition \ref{prop:MatrixMS_entry_bound} and \eqref{eq:bound_for_max_norm_alpha}, the size of a solution in compact representation in Algorithm \ref{alg:PGwithCR} is polynomial in $n$. 

\subsubsection{Case $g \rightarrow \infty$}

Here $T' = \bigO(2^{\bigO(g)})$.
and  
$\rr{n^2+ng}=\bigO(g^{2+\td})$ by (\ref{eq:comp_RRcomputation}) as $C_f =\bigO(g)$, yielding an asymptotic run time of $\bigO\left(g^{4+\td}2^{g(n+1)q^{\td}}\right)=2^{\bigO(g)}$ for Algorithm \ref{alg:PGwithCR}. 
Compact representations of solutions have polynomial height in $g$.

\subsubsection{Case $\deg c \rightarrow \infty$}

The worst case is again given by the upper bound in~\eqref{eq:boundon_prod_vpc+1},
in which vase $\vert S_c \vert =  \bigO(\deg c)$. From (\ref{eq:comp_RRcomputation}), we see that $\rr{n\deg c+ g}=\bigO((\deg c)^2)$, giving an asymptotic complexity of $\bigO\left(2^{n\deg c }(\deg c)^{\omega } \right)$ for Algorithm \ref{alg:PGwithCR}.
When $cO_F$ is inert, this improves to 
$\bigO\left((\deg c)^{\omega}\right)$, which is polynomial in $\deg c$. Here, $\hh(\dvs(cO_F))=\bigO(\deg c)$, so a solution in compact representation has polynomial size in $\deg c$.

\subsubsection{Case $\deg q \rightarrow \infty$}

Here, (\ref{ineq:RS'boundq^2g}) yields an asymptotic complexity of $\bigO(q^{gn+\td})$ for Algorithm \ref{alg:PGwithCR}.


\medskip 

Note that in all cases, Algorithm \ref{alg:PGwithCR} represents an enormous speed-up over Ga\'{a}l-Pohst, and the solutions have far smaller sizes. So introducing compact representation results in a significant gain in time and space efficiency. 

\subsection{Index calculus}\label{sec:comp_PrincCR}
We now
 analyze the expected running time of Algorithm \ref{alg:princwithCR}
Again we assume that $F$ has at least one infinite place of degree 1. 
%
Define
\begin{equation}\label{eq:define_d_Sc}
d_{S_c}:=\max \left\{2^{(\vert S_c \vert -1)(\vert S_c \vert -2)/4}R_S', ~\max_{P\in S_{c,0}} v_P(c)\right\}.
\end{equation}
For any augmented matrix [$M_{S_c}^T \ \ \aug \ \  V]$ appearing in \eqref{eq:matrixeq_NE3princ}, we have $\norm{[M_{S_c}^T \ \ \aug \ \  V] }_{\infty} \le d_{S_c}$, and $\max_{1\le i \le \vert S_c\vert -1} \{\hh(\epsilon_i)\}\le d_{S_c}$. Note that $d_{S_c}$ only depends on $c$ and $F$.


We enumerate all $O_F$-ideals $I$ that divide $cO_F$ as given in \eqref{eq:PGCRformingI}. The number of these $I$ is given by \eqref{eq:number_of_I}. 
For each such~$I$, we must compute five components: the HNF representation of $I$, the norm $\textrm{Norm}_{F/k(x)}(I)$, a solution $\mathcal{X}$ of (\ref{eq:matrixeq_NE3princ}), a vector $v_0$ in the unit lattice that is close to~$\mathcal{X}$, 
and a compact representation of the solution of \eqref{eq:ne_defn} corresponding to $\mathcal{X}$.


Computing the HNF representation of such $I$ 
entails computing products of HNF representations which, by \cite[Theorem~5.2.2]{ATthesis} for example, can be done in 
$\bigO\left( n^7 g^2 q^{\td}\right)$
bit operations. 

Since $c$ is in $k[x]$, each $I$ is integral, 
so Lemma \ref{lem:norm_I_hnf} shows that computing the norm of each $I$ is  
$\bigO\left( \left( n^{2+\td}\deg c + (n \deg c)^{1+\td} \right) q^{\td}\right)$ 
bit operations.

To determine the cost of solving a matrix equations $M_{S_c,0}^T X =V$ as given in (\ref{eq:matrixeq_NE3princ}), we note that
%
\begin{equation*}
    \norm{M_{S_c,0}^T}_\infty = \max_{\substack{P\in {P}_\infty(F)\\ 1\le i \le \vert S_c \vert-1}} v_P(\epsilon_i) , \qquad \norm{V}_\infty \le \max_{P \in {P}_0(F)}v_P(c).
\end{equation*}
by construction. By \cite[Theorem 22]{matrixequation_solv}, the cost of solving $M_{S_c,0}^T X= V$ is $\bigO(\vert S_c \vert ^{\omega +\epsilon}d_{S_c}^{\epsilon})$ bit operations.


To find a vector $v_0$ in the unit lattice closest to $M_{S_c,\infty}^T \mathcal{X}$, we can use the algorithm in \cite{BeckerGamaJoux_SVPCVP}. Its expected run time is $\bigO(2^{0.3774 r})$, and we have 
\[ \norm{M_{S_c,\infty}^T \mathcal{X} - v_0}_\infty \le \max_{1\le i \le r}\left\{{\sum_{j=1}^{r} \frac{1}{2} \vert v_{P_i}(\epsilon_j) \vert}\right\} \le d_{S_c}. \]
Finally, computing a compact representation $\mathbf{t}$ of a solution corresponding to $\mathcal{X}$ 
is done by invoking Algorithm \ref{alg:comprep} on inputs $A= \prod_{i=1}^{\vert S_{c,0}\vert } \mathfrak{p}_i^{(V)_{i}}$ and $V-v_0$, where $\mathfrak{p}_i$ is the prime ideal corresponding to the place $P_i \in S_{c,0}$. We have $\hh(\dvs(A)) \le \hh(\dvs(cO_F))$ and 
$\hh(\text{Norm}(A)) = \deg c$. Thus, computing $\mathbf{t}$ 
takes
\begin{equation*}\label{eq:comp_comprep_minimal}
\bigO\left( gq^{\td} \left( \rr{n\deg c +g }+ \log \left(d_{S_c}+g\right) \rr{n^2+ng}\right)\right)
\end{equation*}
bit operations by Lemma \ref{lem:comp_of_comprep}. 

Putting all these estimates together, we obtain the following asymptotic run time for Algorithm \ref{alg:princwithCR}. 

\begin{thm}\label{lem:comp_alg_princwithCR}
With $d_{S_c}$ defined as in (\ref{eq:define_d_Sc}), Algorithm \ref{alg:princwithCR} solves a norm equation (\ref{eq:ne_defn}) in  
\begin{equation*}\begin{split}
\bigO\biggl(  \left(\prod_{P\in S_{c,0}} (v_{P}(c)+1)\right)\bigl( \vert S_c \vert ^{\omega +\td}d_{s_c}^{\td}+2^{0.3774r}\\ + q^{\td}\left(n^7g^2+n^{2+\td}(\deg c)^{1+\td}+ g \rr{ n\deg c + g} + g \log (d_{S_c}+g) \rr{n^2+ng} \right)\bigr) \biggr) 
\end{split}
\end{equation*}
bit operations, when $n$, $g$, $\deg c$, and $q\rightarrow \infty$. 
\begin{proof}
Each iteration of the loop in steps 6 consists of the five components listed above for an $O_F$-ideal $I$,
and thus takes
\begin{align*}
\bigO(\left(n^7g^2+n^{2+\td}\deg c + (n\deg c)^{1+\td}\right)q^{\td}+ \vert S_c \vert ^{\omega +\td}d_{S_c}^{\td}+2^{0.3774r}\\ + gq^{\td}\left( \rr{n\deg c + g} +  \log (d_{S_c}+g) \rr{n^2+ng}\right))
\end{align*}
bit operations. 
The number of iterations is 
the quantity in \eqref{eq:number_of_I}. Combining all these costs yields the result. 
\end{proof}
\end{thm}

 The compact representations produced by Algorithm \ref{alg:princwithCR} are subject to the same height bounds as those obtained from Algorithm \ref{alg:PGwithCR}. 

We consider the
complexity of Algorithm \ref{alg:princwithCR} when, as before, only one of $n$, $g$, $\deg c$ and $q$ tends to infinity and the others stay fixed. For a bound on $d_{S_c}$ as given in \eqref{eq:define_d_Sc}, we bound $R_F$ by \eqref{ineq:RS'boundq^2g} and use \eqref{eq:boundon_prod_vpc+1}.



\subsubsection{Case $n \rightarrow \infty$}
From (\ref{eq:comp_RRcomputation}), we obtain $\mathsf{RR}(n^2+ng) = \bigO(n^9)$. The upper and lower bounds in \eqref{eq:boundon_prod_vpc+1} 
yield $d_{S_c}= \bigO(2^{n^2})$ and $d_{S_c} = O(1)$, respectively.
The asymptotic complexity of Algorithm \ref{alg:princwithCR} is
    $\bigO \left(  2^{n(0.3774+\deg c)} \right)=2^{n(\deg c+o(1))}$ and $\bigO \left( 2^{0.3774n} \right)$, respectively, i.e.\ exponential. In the first 
case, this is due to the fact that the number of ideals divising $cO_F$ is exponential in $n$, whereas in the second case, the estimate arises from the cost of solving the system \eqref{eq:matrixeq_NE3princ},  

\subsection{Case $g \rightarrow \infty$}
Here, $\mathsf{RR}(n^2+ng) = \bigO(g^{2+\td})$, and $\log (d_{S_c}+g)  = 
\bigO(g^{1+\td})$
giving a polynomial asymptotic complexity of $\bigO\left(  g^{4+\td} \right)$ for Algorithm \ref{alg:princwithCR}
This is asymptotically less than the precomputation of $M_{S_c}$, which is subexponential in $g$ by Lemma \ref{lem:comp_SValMat}. 

\subsection{Case $c \rightarrow \infty$}

From 
(\ref{eq:comp_RRcomputation}), we obtain $\mathsf{RR}(n\deg c +g) = \bigO((\deg c) ^2)$. The bounds on $d_{S_c}$ are the same as in the case $n \rightarrow \infty$. The upper and lower bounds in \eqref{eq:boundon_prod_vpc+1} yields respective asymptotic complexities of $\bigO\left( 2^{n\deg c}(\deg c) ^{\omega+2 +\td}\right)$ and $\bigO((\deg c)^2)$ for Algorithm \ref{alg:princwithCR}; in the latter case, the cost of computing compact representations dominates the overall run time. 

\subsection{Case $q \rightarrow \infty$}

Here, the asymptotic run time of Algorithm \ref{alg:princwithCR}  is $\bigO\left(q^\td\right)$ which is less expensive than the precomputation of $M_{S_c}$, whose cost is given in 

\medskip

Once again, in all cases Algorithm \ref{alg:princwithCR} substantially outperforms Algorithm \ref{alg:PGwithCR}, and Ga\'{a}l-Pohst even more so. This is because the number of ideals tested in Algorithm \ref{alg:princwithCR} is far smaller than the number of elements in the search spaces of the other two algorithms. 

Table \ref{tab:summarycompPICR} summarizes the asymptotic complexities of the Ga\'{a}l-Pohst method, Algorithm \ref{alg:PGwithCR}, and Algorithm \ref{alg:princwithCR}. The complexity estimates in this table are born out by our numerical experiments, described in the next section.

\begin{table}[]
\renewcommand{\arraystretch}{1} 
\begin{center}
\begin{tabular}{|ccc|}
\hline
\multicolumn{1}{|c|}{Ga\'{a}l-Pohst}        & \multicolumn{1}{c|}{Algorithm \ref{alg:PGwithCR}} & Algorithm \ref{alg:princwithCR}  \\ \hline
\multicolumn{3}{|c|}{$n\rightarrow \infty$ (worst)}          \\ \hline

\multicolumn{1}{|c|}{$ 2^{ \bigO\left( n^2 2^{n^2/4}\right)} $}                  & \multicolumn{1}{c|}{$2^{n^3(1+o(1))}$} & {$2^{n(
\deg c + o(1))}$} \\ \hline
\multicolumn{3}{|c|}{$n\rightarrow \infty$, $ cO_F$ prime and inert (best)}  \\ \hline
\multicolumn{1}{|c|}{
$2^{\bigO\left(n^2 2^{n^2/4}\right)}$}        & \multicolumn{1}{c|}{$2^{n^3 (1+o(1))}$} & $\bigO(2^{0.3774 n})$      
\\ \hline
\multicolumn{3}{|c|}{$g\rightarrow \infty$}             \\ \hline
\multicolumn{1}{|c|}{$2^{\bigO(q^g)}$} & \multicolumn{1}{c|}{$2^{\bigO(g)}$} & $\bigO( g^{4+\td})$        \\ \hline
\multicolumn{3}{|c|}{$\deg c \rightarrow \infty$  (worst)}    \\ \hline
\multicolumn{1}{|c|}{$\bigO\left(2^{q^{\td}\deg c} (\deg c)^{\omega}\right)$}      & \multicolumn{1}{c|}{$\bigO(2^{n\deg c}(\deg c)^{\omega})$}                       & $\bigO(2^{n\deg c} (\deg c)^{\omega + 2+\td})$ \\ \hline
\multicolumn{3}{|c|}{$\deg c \rightarrow \infty$, $cO_F$ prime and inert (best)}      \\ \hline
\multicolumn{1}{|c|}{
$\bigO\left(2^{q^\td \deg c}(\deg c)^{\omega}\right)$}  & \multicolumn{1}{c|}{$\bigO((\deg c)^\omega)$}           & $\bigO((\deg c)^2)$ \\ \hline
\multicolumn{3}{|c|}{$q \rightarrow \infty$}            \\ \hline
\multicolumn{1}{|c|}{$q^{\bigO(q^g)}$} & \multicolumn{1}{c|}{$\bigO(q^{gn+\td})$}    & $\bigO(q^{\td})$  
\\ \hline
\end{tabular}\caption{Summary of asymptotic complexity of Ga\'{a}l-Pohst, Algorithm \ref{alg:PGwithCR}, and Algorithm \ref{alg:princwithCR} }\label{tab:summarycompPICR}
\end{center}
\end{table}

\section{Empirical analysis}
The Ga\'{a}l-Pohst method, Algorithm \ref{alg:PGwithCR} and Algorithm \ref{alg:princwithCR}
were all implemented in 
Magma \cite{magma}. 
Our code is available on the first author's GitHub \cite{github_Sumin}. All experiments were performed on an Intel Xeon CPU E7-8891 v4 with 80 64-bit cores at 2.80GHz. 
Selected test results are provided in Figures \ref{fig:n_g=1} and~\ref{fig:g}; an extensive suite of tests and their results can be found in \cite[Chapter 5]{Leem_thesis}. 


All tests and timings were performed on function fields $F/\mathbb{F}_q(x)$, where $q$ is a prime not dividing $n = [F:k(x)]$ and $F$ has at least one infinite place of degree 1. The fields in the selected results presented here all had ideal class number $h_{O_F}=1$. 
Tests were conducted with randomly generated function fields with specified $n$, $g$, $q$, $\deg c$. 
Computations were forcefully terminated when the CPU time required for an algorithm and its precomputation exceeded one day. 
Timings are given in terms of the average number of CPU seconds. 
 

We summarize our observations of our timing tests as follows:
\begin{itemize}
    \item In all test cases, for any norm equation, Algorithm \ref{alg:princwithCR} outperformed Algorithm \ref{alg:PGwithCR} and the Ga\'{a}l-Pohst method.
    \item Our timing results are largely well-aligned with our complexity analysis. Since we only considered worst case complexities, they are not expected to match exactly.
    \item With a set of minimal parameters, $g=1$, $q=3$, $n=1$, $\deg c=1$, the Ga\'{a}l-Pohst method was as fast as Algorithm \ref{alg:PGwithCR} (see Figure \ref{fig:n_g=1}).
    \item As expected from the complexity results, Ga\'{a}l-Pohst slowed down at a much faster rate
    than the other two algorithms as the parameter sizes increased.
    Figure \ref{fig:n_g=1} illustrates this.
    \item The average CPU time taken for precomputating an $S$-unit value matrix was negligible compared to Ga\'{a}l-Pohst and Algorithm \ref{alg:PGwithCR} in all test examples. However, the precomputation took longer than Algorithm \ref{alg:princwithCR} in many cases. This is also not unexpected.
\end{itemize}

Figure \ref{fig:n_g=1} shows the timing results with $n$ varying from $1$ to $10$ and $g=1$, $q=3$, $\deg c=1$ fixed, along with the asymptotic complexities. 
The timing data of Ga\'{a}l-Pohst are only available for $n=2, 4, 5$, because for $n=7$, solving one norm equation already took more than a day due to the huge search space. 
As expected, the shapes of the graphs of the asymptotic complexities and timing results are largely similar, but the actual run times grow slower than the asymptotic complexities.

For all the testing examples, Algorithm \ref{alg:princwithCR} was the fastest. Ga\'{a}l-Pohst was faster than Algorithm \ref{alg:PGwithCR} for some examples with $n=2$, but on average Algorithm \ref{alg:PGwithCR} was faster. From $n=3$ on, Algorithm \ref{alg:PGwithCR} outperformed Ga\'{a}l-Pohst substantially. Ga\'{a}l-Pohst and Algorithm \ref{alg:PGwithCR} are more substantially affected by the growth in $n$, mainly because their search spaces expand doubly exponentially as $n$ grows, while the number of ideals to search in Algorithm \ref{alg:princwithCR} grows exponentially. 

\begin{figure}
\centering
\begin{tikzpicture}
\begin{semilogyaxis}[width=0.9*\textwidth,height=5cm,
axis y line*=left,
xlabel = {$n$},
ylabel = {average CPU seconds},
xmin = 2, xmax = 10,
xtick ={2,3,4,5,6,7,8,9,10},
ymin =  0.007, ymax =1000,
legend entries = {Ga\'{a}l-Pohst, Algorithm \ref{alg:PGwithCR}, Algorithm \ref{alg:princwithCR}},
legend style = {at = {(0.15,-0.25)}, anchor=north}
]
\addplot[color=blue, mark=*] table [y=PG, col sep=comma] {data/n.csv};
\addplot[color=red, mark=square] table [y=PGCR, col sep=comma] {data/n.csv};
\addplot[color=violet, mark=x] table [y=PICR, col sep=comma] {data/n.csv};
\end{semilogyaxis}

\begin{semilogyaxis}[width=0.9*\textwidth,height=5cm,
axis y line*=right,
xlabel = {$n$},
ytick=\empty,
xmin = 2, xmax = 10,
xtick ={2,3,4,5,6,7,8,9,10},
ymin = 4, ymax =10000000000,
legend entries = {Asymptotic complexity of Ga\'{a}l-Pohst, Asymptotic complexity of Algorithm \ref{alg:PGwithCR}, Asymptotic complexity of Algorithm \ref{alg:princwithCR}},
legend style = {at={(0.65,-0.25)}, anchor=north}
]
\addplot[color=blue, domain=2:3, dashed, mark=*]{(1/20000000000)*x^4 *2^(1.58*(x^2 *3^1* 2^(x^2 /4)) + (x^2 - x)/4)};
\addplot[color=red, domain=2:7, dashed, mark=square]{(1/4000)*x^(11)*(x^x)*2^((x^3 - x^2)/4+x + 1.58*x)};
\addplot[color=violet, domain=2:10, dashed, mark=x]{2^(1.3774*x-0.7)};
\end{semilogyaxis}
\end{tikzpicture}  
\captionsetup{justification=centering,margin=1cm}
\caption{Empirical timing results and asymptotic complexities for varying $n$ ($g=1$, $q=3$, $h_{O_F}=1$,  $\deg c=1$, and irreducible $c$)}
 \label{fig:n_g=1}
\end{figure}
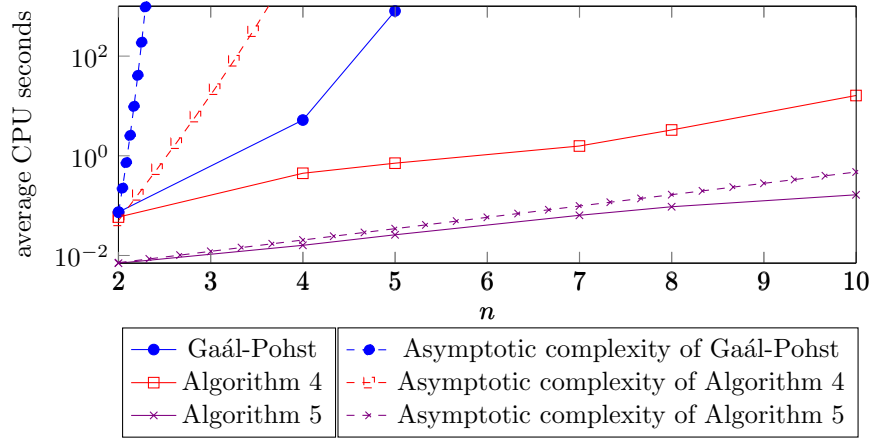

Figure \ref{fig:g} shows the timing results for varying $g$ with $1 \le g \le 18$, with $n=2$, $q=3$, $\deg c=1$. 
Again, there is good agreement between empirical and predicted run times,
No timings for Ga\'{a}l-Pohst resuls are shown 
because the test took too long.
For $g = 1$, Ga\'{a}l-Pohst took 0.059 CPU seconds on average. Already for $g=2$, Ga\'{a}l-Pohst ran more than 1 day.
For $g\ge9$, Algorithm \ref{alg:PGwithCR} ran over a day.
In all test examples for varying $g$, as expected, Algorithm \ref{alg:princwithCR} was the fastest, and the precomputation took longer than the time required by Algorithm \ref{alg:princwithCR} to solve a norm equation.

\begin{figure}
\centering
\begin{tikzpicture}
\begin{semilogyaxis}[width=0.9*\textwidth,height=5cm,
xlabel = {$g$},
ylabel = {average CPU seconds},axis y line*=left,
xmin = 1, xmax = 19, xtick={1,2,4,6,8,10,12,14,16,18},
ymin =  0, ymax =6050,
legend entries = {Algorithm \ref{alg:PGwithCR}, Algorithm \ref{alg:princwithCR}}, legend style = {at={(0.15, -0.25)}, anchor=north}
]
\addplot[color=red, mark=square] table [y=PGCR, col sep=comma] {data/g.csv};
\addplot[color=violet, mark=x] table [y=PICR, col sep=comma] {data/g.csv};
\end{semilogyaxis}
\begin{semilogyaxis}[width=0.9*\textwidth,height=5cm,
xlabel = {$g$},
xmin = 1, xmax = 19, xtick={1,2,4,6,8,10,12,14,16,18},
axis y line*=right,
ytick=\empty,
ymin =  0, ymax =50000000000,
legend entries = { Asymptotic complexity of Algorithm \ref{alg:PGwithCR}, Asymptotic complexity of Algorithm \ref{alg:princwithCR}}, legend style = {at={(0.65, -0.25)}, anchor=north}
]
\addplot[color=red, domain=1:7, dashed, mark=square]
{(2)*x^4 *2^(3.6*x)};
\addplot[color=violet, domain=1:19, dashed, mark=x]
{exp(((x*log2(x))^(1/2))*((2*log2(x))^(-1/2)))+x^4};
\end{semilogyaxis}
\end{tikzpicture}\captionsetup{justification=centering,margin=1cm}
    \caption{Empirical timing results and asymptotic complexities for varying $g$ ($n=2$, $q=3$, $h_{O_F}=1$, $\deg c=1$, irreducible $c$)}
    \label{fig:g}
\end{figure}
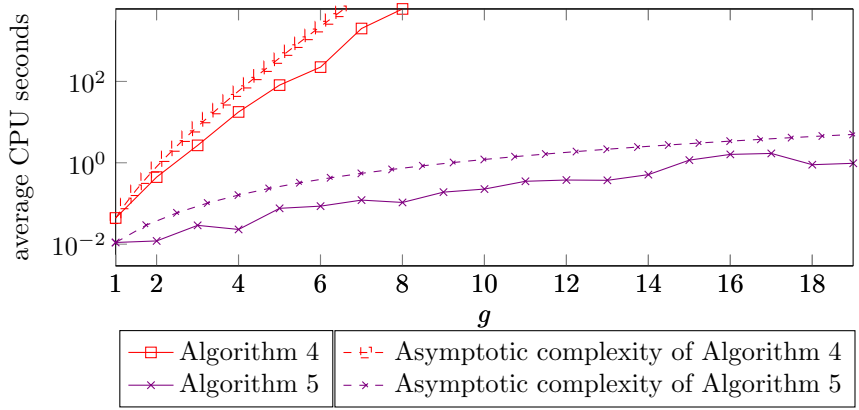

\section{Conclusion}

There are several interesting opportunities for future work to extend the results of our paper.  One open problem is to investigate the performance of our algorithms on norm equations over a global function field $F/\mathbb{F}_q$ with $\gcd (q,n) >1$. Such $F$ were excluded during the analyses in this paper to avoid wild ramification. After running some examples for $F/\mathbb{F}_q$ with $\gcd (q,n) >1$, we noted that the Ga\'{a}l-Pohst method tends to have a larger search space for norm equations over such $F$ compared to those with $\gcd(q,n)=1$. It is unclear whether the larger search spaces are due to wild ramifications. Investigating the reasons may lead to the development of efficient algorithms optimized for norm equations over those $F$. 

A related problem of interest is to find solutions of norm equations in submodules~$M$ of $O_F$, instead of $O_F$. One challenge is that if $M$ has rank less than $n$, we cannot consider solutions in $M$ up to associates. This is because for an element $\alpha\in M$, an element that is associate to $\alpha$ is not guaranteed to belong to $M$.

Another open problem is to develop a version of the algorithm for computing compact representations in \cite{EisentragerHallgren} that does not require $F$ to have an infinite place of degree 1.  An infinite place of degree 1 is required in Algorithm \ref{alg:comprep} to compute a correct compact representation equal to an element $\alpha$ and not just close to $\alpha$. It is unclear how to ensure thus accuracy when there is no infinite place of degree 1.

Finally, given the successful results of this paper, it seems promising to investigate solving different Diophantine equations over global function fields using compact representations.

\bibliographystyle{abbrv}
\bibliography{ref.bib} 
\end{document}